\documentclass[11pt,a4paper]{article}

\usepackage[utf8]{inputenc}
\usepackage[T1]{fontenc}
\usepackage{amsmath,amssymb,amsthm}
\usepackage{geometry}
\usepackage{xcolor}
\usepackage{lmodern}
\usepackage{hyperref}
\usepackage{microtype}
\usepackage{graphicx}
\usepackage{caption} 
\newcommand{\spec}{\mathrm{spec}}
\newcommand{\keywords}[1]{\par\smallskip\noindent\textbf{Keywords: }#1}

\newtheorem{definition}{Definition}
\newtheorem{proposition}{Proposition}
\newtheorem{remark}{Remark}
\newtheorem{theorem}{Theorem}[section]
\newtheorem{corollary}[theorem]{Corollary}
\begin{document}

\title{Generalizing Lattice Structures to Hypergraphs: Spectra of Clique and Hyperedge-based Laplacians}

\author{Eleonora Andreotti\\
Fondazione Bruno Kessler (FBK), Povo (TN), Italy\\
eandreotti@fbk.eu}
\date{}
\maketitle
\begin{abstract}
Lattice structures play a central role in spectral graph theory, offering analytical insight into diffusion, synchronization, and transport processes on regular discrete spaces. 
While their spectral properties are completely characterized in the classical graph setting, 
an extension to hypergraphs, where interactions involve more than two nodes, remains largely unexplored in the matrix-based formulation. 
In this work, we generalize the notion of a lattice to the hypergraph framework 
and study its Laplacian spectra under two alternative definitions: 
the \emph{clique Laplacian}, obtained through pairwise projection, 
and the \emph{hyperedge-based Laplacian}, defined via normalized hyperedge incidences. 
For both definitions, we derive the corresponding Laplacian matrices, 
analyze their eigenvalue spectra, and discuss how they reflect 
the underlying topological and dynamical structure of the hyperlattice. 
Our main result is a theorem giving a full spectral characterization in the periodic case, 
together with a Toeplitz-type open analogue whose spectrum retains a separable trigonometric structure. 
The obtained eigenvalues are expressed explicitly in terms of the hyperedge size, the number of directional 
families, and the lattice side length, thereby capturing how the geometry of higher-order interactions 
shapes the spectral structure.
\end{abstract}

\section{Introduction}
Hypergraphs generalize classical graphs by allowing hyperedges that 
connect an arbitrary number of vertices.
Formally, a hypergraph $\mathcal{H} = (V,\mathcal{E})$ consists of a finite
vertex set $V$ and a family of non-empty subsets
$\mathcal{E} \subseteq 2^V$, called hyperedges.
When $|e|=2$ for all $e \in \mathcal{E}$, one recovers the usual notion
of a graph.\\
Several mathematical formulations have been proposed to extend
spectral graph theory to this setting.
A first approach represents hypergraphs through
higher-order adjacency tensors, leading to spectral definitions based on
homogeneous polynomial eigenvalue equations
(see, e.g., Cooper and Dutle~\cite{CooperDutle2012},
Pearson and Zhang~\cite{PearsonZhang2014},
and Shao~\cite{Shao2013}).\\
In the matrix setting, one seeks to project the higher-order structure
of $\mathcal{H}$ onto a weighted graph that preserves, as much as possible,
the combinatorial or dynamical properties of the original hypergraph.
However, this projection is not unique:
different reduction schemes lead to distinct adjacency and Laplacian operators,
each emphasizing a particular aspect of the underlying structure.\\
When defining a diffusion or random-walk process on a hypergraph, for instance, 
it becomes necessary to specify how two nodes exchange the diffusing quantity through the hyperedges they share.  
\textcolor{black}{Unlike in simple graphs, the number of neighbors of a node does not coincide with the number of hyperedges it belongs to, and this leads to two natural interpretations of the Laplacian \cite{Mulas2022RandomWalks}: one based on pairwise adjacency and the other on hyperedge incidence.}
The first interpretation follows the idea of preserving pairwise connectivity, 
treating each hyperedge as a complete clique among its nodes.  
This viewpoint underlies the approach proposed by Carletti \emph{et al.}~\cite{Carletti2020RandomWalks}, 
where diffusion and random-walk dynamics are defined through such pairwise projections.  
In this framework, the Laplacian is constructed as in the classical graph case, with node degrees corresponding to the number of projected cliques incident to each node.  
Nevertheless, by treating every hyperedge as a clique, 
the intrinsic higher-order structure of the hypergraph is lost, 
and the node degree becomes a measure of pairwise adjacency rather than hyperedge membership.\\
A conceptually different approach, aimed at maintaining the genuine higher-order structure of the hypergraph, 
was introduced by Banerjee~\cite{Banerjee2021Spectrum}.  
Here, the focus shifts from pairwise relations to hyperedge-based interactions:  
the degree on the diagonal is defined as the number of hyperedges incident to each node, 
while the off-diagonal entries depend on the shared hyperedges between two nodes, 
weighted by the inverse of the hyperedge cardinality minus one,  
\[
A_{uv} = \sum_{e \ni u,v} \frac{1}{|e|-1}, \qquad L = D - A.
\]
The two definitions above provide complementary matrix formulations
of the hypergraph Laplacian.
Both preserve the key spectral properties of the graph Laplacian, such as
positive semidefiniteness and the correspondence between the multiplicity
of the zero eigenvalue and the number of connected components, 
while capturing different aspects of higher-order connectivity.
The clique-based operator emphasizes pairwise adjacency among nodes,
whereas the hyperedge-based operator encodes interactions mediated
through shared hyperedge membership.
\textcolor{black}{A discussion of how the clique-expanded and hyperedge-based Laplacians arise as specific instances of broader classes of operators associated with random walks on hypergraphs, together with conditions under which the resulting dynamics coincide, is provided by Mulas et al.~\cite{Mulas2022RandomWalks}}.\\

Since in hypergraphs the definition of adjacency and Laplacian operators requires modelling choices 
that do not arise in ordinary graphs, even highly symmetric hypergraphs may exhibit structural and 
dynamical behaviours that vary depending on the chosen formulation. \\
Lattice graphs are the paradigmatic example of a regular, translation-invariant interaction structure. 
The mathematical and physical importance of lattices rests on two features that are tightly coupled in 
the graph setting: regularity of the degree and locality of interactions. Each node in the classical 
$\mathbb{Z}^m$ lattice is adjacent to its $2m$ nearest neighbours, and this homogeneous, pairwise structure 
underpins both its spectral tractability and its role as a reference model for more general systems.
Their symmetry makes them analytically tractable: Laplacian spectra admit closed-form expressions 
via discrete Fourier analysis~\cite{Chung1997Spectral,Merris1994Laplacian,MoharWoess1989Spectra}, and many classical results in spectral graph theory and statistical 
physics are derived precisely in this setting. In particular, lattice models underpin the study of 
diffusion, synchronization, and interacting-particle systems such as the Ising and Potts models~\cite{Ising1925,Onsager1944,Spohn1991},
providing a clean reference case against which more general behaviours can be compared.\\
This naturally raises the problem of constructing a hypergraph analogue of the classical lattice, 
one that preserves translational symmetry while allowing for multiway interactions.
Developing such a structure would provide a higher-order counterpart to the role played by 
$\mathbb{Z}^m$ in graph theory, offering a tractable model for studying diffusion and dynamics 
in systems where interactions involve more than two nodes.\\
Despite their central role in the spectral analysis of graphs,
lattice structures have not yet been systematically generalized to
the hypergraph setting.
The notion of regularity, which in graphs is defined by the node degree,
and that of locality, associated with pairwise adjacency, become
non-trivial when interactions involve more than two nodes.
In a hypergraph, each node may participate in multiple hyperedges of
different cardinalities, and the number of neighboring nodes no longer
coincides with the number of incident hyperedges.
As a consequence, extending the concept of a lattice to hypergraphs
requires specifying how hyperedges are distributed and how they overlap
in a way that preserves translational symmetry along each coordinate
direction.\\

This naturally arises the construction of a hypergraph analogue of the classical lattice. 
In the present work, we introduce such a family of hyperlattices on the discrete torus, 
define their clique-based and hyperedge-based Laplacians, and derive explicit spectral 
expressions for both periodic and open-boundary realizations. In doing so, we obtain a 
higher-order counterpart of classical lattice models whose Laplacian spectra remain 
fully tractable under either boundary condition.

\medskip
\noindent

\textbf{Outline of the paper.}
Section~\ref{sec:laplacians} recalls the definitions of adjacency, degree, and Laplacian operators for hypergraphs, emphasizing the distinction between the clique-based and hyperedge-based formulations, together with their normalized counterparts.\\
Section~\ref{sec:lattice} revisits classical lattice graphs and the
$\ell$-lattice hypergraph introduced by Andreotti and Mulas~\cite{AndreottiMulas2022Signless},
highlighting how row/column hyperedges induce natural matrix formulations.\\
Section~\ref{sec:llattice} derives the spectra of the clique
and hyperedge-based Laplacians for the $\ell$-lattice hypergraph.
This provides a first concrete comparison between the two operators on
a regular higher-order structure with uniform hyperedge cardinality.\\
Section~\ref{sec:2mfamhyper} introduces a broader class of hyperlattices obtained by combining 
$m$ directional families of hyperedges on the discrete torus $\mathbb{Z}_\ell^m$, one for each 
coordinate axis. This construction generalizes the $\ell$-lattice in $\mathbb{Z}_\ell^2$ 
considered in Section~\ref{sec:lattice}, and it admits a closed-form spectral characterization 
under both Laplacian definitions.\\
Section~\ref{sec:2mrhyper} further extends the construction to
hyperlattices on $\mathbb{Z}_\ell^m$ with hyperedges of arbitrary cardinality
$r$, decoupling the interaction range from the lattice size.
For these models, we derive explicit formulas for the spectra of the
clique-based and hyperedge-based Laplacians.
We also compare periodic and open-boundary realizations, and introduce
a Toeplitz-type open hyperlattice whose spectrum admits a separable, 
fully explicit representation.
This section contains our main spectral results, expressing the eigenvalues
in terms of the lattice side length, the dimension, and the hyperedge size.\\
\textcolor{black}{Finally, Section~\ref{sec:ising} collects the main implications of our results, connects them with higher-order Ising models on hypergraphs, and outlines directions for future research.}

\medskip
\section{Laplacian operators on hypergraphs}
\label{sec:laplacians}

In this section we recall the Laplacian formulations that will be employed
throughout the paper.
Let $\mathcal{H} = (V,H)$ be a finite, undirected hypergraph
with node set $V = \{1,\dots,n\}$ and hyperedge family
$H = \{e_1,\dots,e_m\}$, where each $e_j \subseteq V$.
We denote by $|e_j|$ the cardinality of the $j$-th hyperedge.

\subsection{Adjacency and degree matrices}

For any pair of nodes $u,v \in V$, we define the \emph{pairwise weight}
\[
A_{uv} =
\begin{cases}
  0, & u = v,\\[4pt]
  \displaystyle
  \sum_{e \in H \,:\, u,v \in e} w_e \, \omega_{uv}(e),
  & u \neq v,
\end{cases}
\]
where $w_e > 0$ is the weight assigned to the hyperedge $e$ and
$\omega_{uv}(e)$ specifies how the contribution of $e$ is distributed
among its node pairs.
The degree of node $u$ is then $d_u = \sum_{v} A_{uv}$,
and the degree matrix is $D = \mathrm{diag}(d_1,\dots,d_n)$.\\
Different definitions of $\omega_{uv}(e)$ lead to distinct Laplacian operators,
each capturing a different notion of adjacency within the hypergraph.

\subsection{Clique Laplacian}

The \emph{clique Laplacian} is obtained by projecting each hyperedge onto
a complete clique among its nodes, assigning unit weight to every pair:
\[
\omega_{uv}^{(\mathrm{cl})}(e) = 1
\qquad \forall\, u,v \in e,~u\neq v.
\]
The corresponding adjacency and Laplacian matrices are
\[
A_{\mathrm{cl}}(u,v) = \sum_{e \ni u,v} w_e, 
\qquad
L_{\mathrm{cl}} = D_{\mathrm{cl}} - A_{\mathrm{cl}}.
\]
The degree of node~$u$ in this projection, denoted $d_u^{(\mathrm{cl})}$, is
\[
d_u^{(\mathrm{cl})}
= \sum_{v \neq u} A_{\mathrm{cl}}(u,v)
= \sum_{e \ni u} w_e \, (|e|-1).
\]
Hence, in the clique-based formulation each hyperedge~$e$
contributes $(|e|-1)$ to the degree of every node it contains.
The degree matrix is
\[
D_{\mathrm{cl}} = \mathrm{diag}\!\left(d_1^{(\mathrm{cl})}, \dots, d_n^{(\mathrm{cl})}\right).
\]
When all hyperedges have equal weight $w_e = 1$, 
the degree of node~$u$ equals the number of other nodes
that share at least one hyperedge with~$u$.
This corresponds to the ``clique expansion'' introduced by
Carletti \emph{et al.}~\cite{Carletti2020RandomWalks}.
\medskip

\subsection{Hyperedge-based Laplacian}\label{subsec:hyperlap}

Alternatively, one may preserve the hyperedge structure explicitly by
distributing the contribution of each hyperedge uniformly among its pairs.
Following Banerjee~\cite{Banerjee2021Spectrum},
we define
\[
\omega_{uv}^{(\mathrm{h})}(e) = \frac{1}{|e|-1},
\]
so that the total contribution of each hyperedge~$e$
to the degree of any node it contains equals~$w_e$.
The resulting adjacency and Laplacian matrices are
\[
A_{\mathrm{h}}(u,v) = \sum_{e \ni u,v} \frac{w_e}{|e|-1},
\qquad
L_{\mathrm{h}} = D_{\mathrm{h}} - A_{\mathrm{h}}.
\]

The associated node degree, denoted $d_u^{(\mathrm{h})}$, is
\[
d_u^{(\mathrm{h})}
= \sum_{v \neq u} A_{\mathrm{h}}(u,v)
= \sum_{e \ni u} w_e.
\]
Thus, in the hyperedge-based Laplacian each hyperedge
contributes one unit to the degree of every node it contains,
independently of its size.
The degree matrix is
\[
D_{\mathrm{h}} = \mathrm{diag}\!\left(d_1^{(\mathrm{h})}, \dots, d_n^{(\mathrm{h})}\right).
\]
This definition preserves the key spectral properties of graph Laplacians:
it is symmetric, positive semidefinite, and the multiplicity of the zero
eigenvalue equals the number of connected components of~$\mathcal{H}$.

\subsection{Normalized Laplacians}

For both definitions, one can introduce normalized versions analogous
to those used in spectral graph theory.
Let $L$ denote either $L_{\mathrm{cl}}$ or $L_{\mathrm{h}}$, 
and let $D$ be the corresponding degree matrix.
We define:
\begin{align*}
L_{\mathrm{rw}} &= D^{-1} L = I - D^{-1}A, 
&&\text{(random-walk Laplacian)},\\[4pt]
L_{\mathrm{sym}} &= D^{-1/2} L D^{-1/2} = I - D^{-1/2} A D^{-1/2}, 
&&\text{(symmetric normalized Laplacian)}.
\end{align*}
Both operators are positive semidefinite and share the same spectrum
up to similarity.
The random-walk version naturally describes discrete diffusion processes
and stationary probabilities, whereas the symmetric form is more convenient
for spectral analysis.

\begin{remark}
In regular hypergraphs where all nodes have equal degree $d$,
the normalized Laplacians simplify to
\[
L_{\mathrm{rw}} = \frac{1}{d} L,
\qquad
L_{\mathrm{sym}} = \frac{1}{d} L,
\]
and hence coincide (up to scaling) with the combinatorial Laplacian.
\end{remark}
\begin{proposition}[Fundamental spectral properties]
\label{prop:spectral_properties}
Let $\mathcal{H} = (V,H)$ be a connected, undirected, weighted hypergraph,
and let $L$ denote either the clique Laplacian $L_{\mathrm{cl}}$ 
or the hedge Laplacian $L_{\mathrm{h}}$, as defined above.
Then:
\begin{enumerate}
  \item $L$ is symmetric and positive semidefinite.
  \item The smallest eigenvalue of $L$ is $\lambda_0 = 0$,
        with associated eigenvector $\mathbf{1} = (1,\dots,1)^\top$.
  \item The multiplicity of the zero eigenvalue equals the number
        of connected components of $\mathcal{H}$.
  \item All eigenvalues of the normalized Laplacians
        $L_{\mathrm{rw}} = I - D^{-1}A$ and
        $L_{\mathrm{sym}} = I - D^{-1/2} A D^{-1/2}$
        are real and lie in the interval $[0,2]$.
  \item If $\mathcal{H}$ is $d$-regular (i.e., $d_u = d$ for all $u \in V$),
        then the spectra of $L$, $L_{\mathrm{rw}}$ and $L_{\mathrm{sym}}$
        are identical up to the scaling factor $1/d$.
\end{enumerate}
\end{proposition}

\begin{proof}[Sketch of the proof]
Symmetry and positive semidefiniteness follow from the quadratic form
\[
x^\top L x
= \frac{1}{2} \sum_{u,v} A_{uv} (x_u - x_v)^2 \ge 0,
\]
which holds for both Laplacian definitions since $A_{uv} \ge 0$ and $A_{uv}=A_{vu}$.
The constant vector $\mathbf{1}$ is always in the kernel of $L$
because each row of $L$ sums to zero.
Connectivity arguments identical to those for graph Laplacians
ensure that the multiplicity of the zero eigenvalue equals
the number of connected components.
For the normalized operators, $L_{\mathrm{sym}}$ is similar to
$L_{\mathrm{rw}}$ via the transformation
$L_{\mathrm{sym}} = D^{1/2} L_{\mathrm{rw}} D^{-1/2}$,
which implies that they share the same spectrum,
and all eigenvalues lie in $[0,2]$ as shown in
Banerjee~\cite{Banerjee2021Spectrum}.
Finally, in the regular case $D = d I$, so $L_{\mathrm{rw}} = L_{\mathrm{sym}} = L/d$.
\end{proof}

\section{The lattice in graph and hypergraph theory}\label{sec:lattice}

The lattice is among the most fundamental and extensively studied graph structures.
Its regularity, symmetry, and spatial embedding make it a canonical model 
for diffusion, vibration, and transport phenomena on networks.
In the classical graph setting, the $l\times l$ lattice (or grid graph)
is defined as the Cartesian product of two path graphs $P_l \square P_l$, 
where each node is connected to its nearest horizontal and vertical neighbors.
The corresponding Laplacian spectrum is known in closed form and can be expressed as
\[
\lambda_{k,h} = 2\!\left(2 - \cos\frac{\pi k}{l+1} - \cos\frac{\pi h}{l+1}\right),
\qquad k,h = 1,\dots,l,
\]
for the open lattice (with free boundaries),
and
\[
\lambda_{k,h} = 4 - 2\cos\frac{2\pi k}{l} - 2\cos\frac{2\pi h}{l},
\qquad k,h = 0,\dots,l-1,
\]
for the periodic (toroidal) lattice.
These eigenvalues describe, respectively, the discrete normal modes 
of a vibrating membrane or the steady states of a diffusion process on a regular grid.
From a physical perspective, the lattice Laplacian captures
how local interactions between neighboring sites 
give rise to collective behavior such as spatial synchronization,
energy dispersion, or diffusive relaxation.

\medskip

While the graph lattice has been thoroughly characterized,
its generalization to hypergraphs is less well understood.
In \cite{AndreottiMulas2022Signless},
Andreotti and Mulas introduced the concept of an $l$-lattice hypergraph
and analytically studied its spectrum.
In particular, they considered the \emph{signless normalized Laplacian}
operator, and derived its eigenvalues in closed form
(Proposition~5.11 in~\cite{AndreottiMulas2022Signless}).
Given $l \in \mathbb{N}$, $l \ge 2$, they define the $l$-lattice
as the hypergraph $\mathcal{H} = (V,H)$ on $l^2$ nodes and $2l$ hyperedges
that can be represented as an $l \times l$ grid, where
each hyperedge corresponds to a full row or a full column of nodes.
In this construction, every node belongs to exactly two hyperedges:
one horizontal and one vertical.

\medskip

However, just as the notion of the Laplacian on hypergraphs
admits multiple generalizations, depending on whether node degrees
are defined through pairwise adjacencies or hyperedge incidences,
the lattice structure itself can be extended according to different modeling choices.
One may define adjacency based on node co-membership within the same hyperedge
(clique expansion), or instead based on shared hyperedges with appropriate normalization.
These alternative formulations yield distinct Laplacian operators,
each preserving specific structural or dynamical properties of the hyperlattice.

\section{Spectrum of the hypergraph $l$-lattice}\label{sec:llattice}

We now consider the $l$-lattice hypergraph introduced in
Andreotti and Mulas~\cite{AndreottiMulas2022Signless}.
Recall that it is the hypergraph $\mathcal{H}=(V,H)$ on $l^2$ nodes
arranged in an $l\times l$ grid, with exactly $2l$ hyperedges:
the $l$ rows and the $l$ columns of the grid.
Every node belongs to exactly two hyperedges, one horizontal and one vertical,
and all hyperedges have the same cardinality $l$ (see Fig.~\ref{fig:lattice5} for $l=5$).

In this setting, two natural Laplacian operators arise.
\begin{figure}[ht]
    \centering
    \includegraphics[width=.6\linewidth]{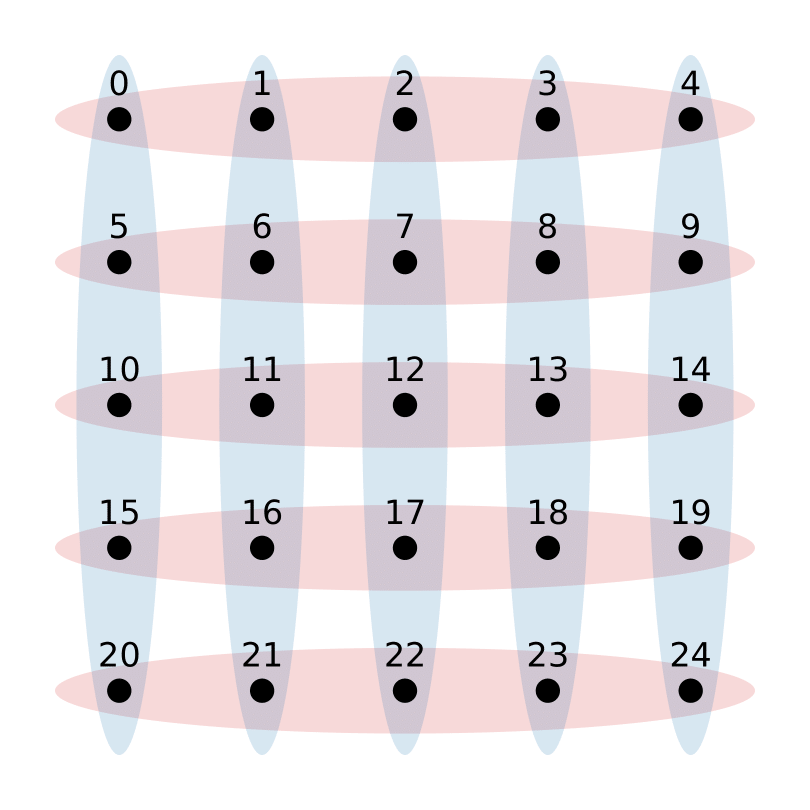}
    \caption{The $5$-lattice hypergraph $\mathcal{H}=(V,H)$.
  nodes are arranged on a $5\times5$ grid.
  Each horizontal (red) and vertical (blue) region 
  represents a hyperedge containing all nodes in a row or column, respectively.
  Every node belongs to exactly two hyperedges: one horizontal and one vertical.
 }
    \label{fig:lattice5}
\end{figure}
\subsection{Clique Laplacian}

The first operator is obtained by clique expansion:
each hyperedge is replaced by a complete graph on its $l$ nodes,
with unit weight on every pair.
Enumerating the nodes as $(r,c)$ with $r,c=0,\dots,l-1$,
this produces exactly the graph on the $l\times l$ grid
in which two nodes are adjacent if and only if they lie in the same row
or in the same column.
In matrix form, the adjacency can be written as
\[
A_{\mathrm{cl}} = (J_l - I_l)\otimes I_l \;+\; I_l \otimes (J_l - I_l),
\]
where $J_l$ is the $l\times l$ all-ones matrix and $I_l$ the identity.
Since each node is adjacent to $(l-1)$ nodes in its row
and to $(l-1)$ nodes in its column, the degree is constant and equal to
$2(l-1)$.
Therefore the clique Laplacian is
\[
L_{\mathrm{cl}} = 2(l-1) I_{l^2} - A_{\mathrm{cl}}.
\]

\begin{proposition}[Spectrum of the clique Laplacian on the $l$-lattice]
The eigenvalues of $L_{\mathrm{cl}}$ are
\[
\lambda_0 = 0 \quad (\text{multiplicity } 1), \qquad
\lambda_1 = l \quad (\text{multiplicity } 2(l-1)), \qquad
\lambda_2 = 2l \quad (\text{multiplicity } (l-1)^2).
\]
\end{proposition}

\begin{proof}[Sketch of the proof]
The matrix $J_l - I_l$ has eigenvalues $l-1$ (with multiplicity $1$)
and $-1$ (with multiplicity $l-1$).
Hence, by standard properties of Kronecker sums, the adjacency matrix
\(
A_{\mathrm{cl}} = (J_l - I_l)\otimes I_l + I_l\otimes (J_l - I_l)
\)
has eigenvalues given by all possible sums
\[
(l-1) + (l-1) = 2(l-1) \quad (1\text{ time}),
\]
\[
(l-1) + (-1) = l-2 \quad (l-1\text{ times}),
\]
\[
(-1) + (l-1) = l-2 \quad (l-1\text{ times}),
\]
\[
(-1) + (-1) = -2 \quad ((l-1)^2 \text{ times}).
\]
Since $L_{\mathrm{cl}} = 2(l-1)I - A_{\mathrm{cl}}$,
the eigenvalues of $L_{\mathrm{cl}}$ are obtained by subtracting the
above values from $2(l-1)$, which yields the statement.
\end{proof}

\subsection{Hyperedge-based Laplacian}

A second, hyperedge-based operator can be defined as follows.
Given a hyperedge $e$ of cardinality $l$ and unit weight,
every pair of nodes in $e$ receives weight $1/(l-1)$,
so that the total contribution of $e$ to the degree of each of its nodes is $1$.
Since in the $l$-lattice each node belongs to exactly two hyperedges,
the degree is identically $2$.
Accordingly, the hedge Laplacian is
\[
L_{\mathrm{h}} = 2 I_{l^2} - \frac{1}{l-1} A_{\mathrm{cl}}.
\]

\begin{proposition}[Spectrum of the hyperedge-based Laplacian on the $l$-lattice]
The eigenvalues of $L_{\mathrm{h}}$ are
\[
\mu_0 = 0 \quad (\text{multiplicity } 1), \qquad
\mu_1 = \frac{l}{l-1} \quad (\text{multiplicity } 2(l-1)), \qquad
\mu_2 = \frac{2l}{l-1} \quad (\text{multiplicity } (l-1)^2).
\]
\end{proposition}

\begin{proof}[Sketch of the proof]
Let $\alpha$ be an eigenvalue of $A_{\mathrm{cl}}$.
Then $\lambda = 2 - \frac{1}{l-1}\alpha$ is an eigenvalue of
$L_{\mathrm{h}} = 2I - \frac{1}{l-1}A_{\mathrm{cl}}$
with the same multiplicity.
Substituting the three values $\alpha \in \{2(l-1),\, l-2,\, -2\}$
with their respective multiplicities yields the formula.
\end{proof}
\begin{remark}[Comparison with the classical lattice graph]
Notice that the clique Laplacian on the $l$-lattice hypergraph behaves spectrally
like a ``rank--two'' perturbation of the classical 2D lattice graph.
In the graph case (grid with only 4-neighbourhood and no hyperedges),
the Laplacian eigenvalues form a two-parameter family of trigonometric terms,
depending on the horizontal and vertical frequencies.
In the hypergraph case considered here, however, each row and each column
induces a complete interaction among its nodes, which collapses
the whole spectrum into only three distinct eigenvalues,
\[
\sigma(L_{\mathrm{cl}}) = \{\,0,\; l,\; 2l\,\},
\]
with multiplicities $1$, $2(l-1)$, and $(l-1)^2$, respectively.
This shows that the hyperedge-induced coupling is much stronger than
simple nearest-neighbour coupling, and it rigidifies the system
to a finite set of collective modes.
\end{remark}

\begin{remark}[Effect of hyperedge normalization]
The hyperedge Laplacian $L_{\mathrm{h}}$ differs from the clique Laplacian
only by a uniform rescaling of the pairwise weights inside each hyperedge.
Yet this rescaling is not innocuous: it changes the degree from $2(l-1)$
to the constant value $2$, and correspondingly shifts the two nonzero
eigenvalues from $\{l, 2l\}$ to
\[
\left\{\frac{l}{l-1}, \, \frac{2l}{l-1}\right\}.
\]
In other words, the hyperedge Laplacian preserves the combinatorial
information ``how many hyperedges contain a node''
rather than ``how many nodes are pairwise adjacent to it''.
This illustrates, on a minimal hypergraph where all hyperedges have
the same size, that the two most common Laplacian constructions for
hypergraphs encode genuinely different diffusion scales.
\end{remark}

\section{Hyperlattices with multi-directional families}
\label{sec:2mfamhyper}
As discussed in the previous section, 
the $l$-lattice hypergraph introduced by 
Andreotti and Mulas~\cite{AndreottiMulas2022Signless}
preserves the visual grid structure of the classical two-dimensional lattice:
the nodes are arranged in an $l\times l$ array, and the hyperedges
correspond exactly to the $l$ rows and the $l$ columns of the grid.
Consequently, each node is contained in precisely two hyperedges,
one horizontal and one vertical, so that the \emph{hyperedge-based degree} is
\[
d^{\mathrm{(h)}}(v) = 2,
\quad \text{for all } v \in V,
\]
and, since each hyperedge has cardinality $l$, the corresponding
\emph{clique-degree} is
\[
d^{\mathrm{(cl)}}(v) = 2(l-1).
\]
We now reinterpret this construction in terms of
\emph{directional families of hyperedges},
which will provide a convenient framework for subsequent generalizations.

\begin{definition}[Directional family in $\mathbb{Z}^2_l$]
Let $V = \mathbb{Z}_l \times \mathbb{Z}_l$ be the discrete two-dimensional torus.
A \emph{directional family of hyperedges} is a collection
$\mathcal{F} = \{ e_1, \ldots, e_l \}$ such that:
\begin{enumerate}
    \item all hyperedges in $\mathcal{F}$ are parallel and disjoint,
          and their union covers $V$;
    \item each node belongs to exactly one hyperedge in the family.
\end{enumerate}
Each family thus represents a set of parallel ``lines'' on the torus,
and defines one independent direction of interaction.
\end{definition}
\begin{definition}[$\ell$-hyperlattice in $\mathbb{Z}_\ell^2$]
Let $V = \mathbb{Z}_\ell \times \mathbb{Z}_\ell$ be the discrete two-dimensional torus. 
An \emph{l-hyperlattice} on $V$ is defined as the hypergraph 
$\mathcal{H} = (V, \mathcal{E})$ obtained as the union of two 
directional families of parallel hyperedges:
\[
\mathcal{E} = \mathcal{F}_{\mathrm{row}} \cup \mathcal{F}_{\mathrm{col}},
\qquad 
\mathcal{F}_{\mathrm{row}} = \{ H_{\mathrm{row}}(r) \}_{r=0}^{\ell-1}, 
\quad 
\mathcal{F}_{\mathrm{col}} = \{ H_{\mathrm{col}}(c) \}_{c=0}^{\ell-1},
\]
where
\[
H_{\mathrm{row}}(r) = \{ (r, c) : c \in \mathbb{Z}_\ell \}, 
\qquad 
H_{\mathrm{col}}(c) = \{ (r, c) : r \in \mathbb{Z}_\ell \}.
\]
Each node belongs to exactly one hyperedge from each directional family, 
hence to two hyperedges in total. 
Consequently, the hyperedge-based and clique-based degrees are constant and equal to
\[
d^{(h)}(v) = 2, 
\qquad 
d^{(\mathrm{cl})}(v) = 2(\ell - 1),
\quad \forall v \in V.
\]
\end{definition}



\subsection{$\ell$-hyperlattices in $\mathbb{Z}^m_{\ell}$}
We now extend the construction of the $\ell$-hyperlattice in~$\mathbb{Z}_\ell^2$
to higher dimensions, following the same geometric rationale 
as for the classical $m$-dimensional lattice graph in~$\mathbb{Z}^m$.
\begin{definition}[Directional family in $\mathbb{Z}^m_\ell$]
Let $V = \mathbb{Z}_\ell^m$ be the $m$-dimensional discrete torus.
A \emph{directional family of hyperedges} is a collection
\[
\mathcal{F}
= \{ e_{x_{\hat{i}}} : x_{\hat{i}} \in \mathbb{Z}_\ell^{\,m-1} \},
\]
indexed by the $(m-1)$--tuples $x_{\hat{i}} = (x_1,\dots,x_{i-1},x_{i+1},\dots,x_m)$,
such that
\begin{enumerate}
    \item each hyperedge is the coordinate line
    \[
    e_{x_{\hat{i}}}
    = \{\, (x_1,\dots,x_{i-1},t,x_{i+1},\dots,x_m) : t \in \mathbb{Z}_\ell \,\};
    \]
    \item hyperedges in $\mathcal{F}$ are pairwise disjoint and
          their union equals $V$;
    \item each vertex belongs to exactly one hyperedge of the family.
\end{enumerate}
A directional family thus represents all coordinate lines parallel to axis $i$
in the torus $\mathbb{Z}_\ell^m$.
\end{definition}
\begin{definition}[$\ell$-hyperlattice in $\mathbb{Z}_\ell^m$]
Let $\ell \ge 2$ and $m \ge 1$.  
The \emph{$\ell$-hyperlattice} in the discrete $m$-dimensional torus
\[
V = \mathbb{Z}_\ell^m
= \{ (x_1,\dots,x_m) : x_i \in \{0,\dots,\ell-1\} \}
\]
is the hypergraph $\mathcal{H}_\ell = (V,\mathcal{E})$ obtained as follows.

For each coordinate direction $i \in \{1,\dots,m\}$, consider the directional
family $\mathcal{F}_i$ consisting of all coordinate lines parallel to axis $i$:
\[
\mathcal{F}_i
=\bigl\{
L_i(x_{\hat{i}}) 
= 
\{(x_1,\dots,x_{i-1},t,x_{i+1},\dots,x_m) : t\in\mathbb{Z}_\ell\}
\;:\;
x_{\hat{i}}\in\mathbb{Z}_\ell^{\,m-1}
\bigr\}.
\]
Each hyperedge $L_i(x_{\hat{i}})$ contains exactly $\ell$ vertices and the
hyperedges in $\mathcal{F}_i$ are pairwise disjoint, covering $V$.

The hyperedge set of the $\ell$-hyperlattice is
\[
\mathcal{E}=\bigcup_{i=1}^m \mathcal{F}_i.
\]

Every vertex belongs to exactly one hyperedge in each family $\mathcal{F}_i$,
hence to $m$ hyperedges in total.  
Therefore the degrees are constant:
\[
d^{(h)}(v)=m,
\qquad
d^{(\mathrm{cl})}(v)=m(\ell-1),
\qquad\forall v\in V.
\]
\end{definition}

\begin{remark}
The $\ell$-hyperlattice in $\mathbb{Z}_\ell^m$ is the natural 
$m$-dimensional extension of the $\ell$-hyperlattice in 
$\mathbb{Z}_\ell^2$.  
Each coordinate axis of the discrete torus contributes one 
independent directional family, consisting of all coordinate 
lines parallel to that axis.  
In this way, every vertex belongs to exactly one hyperedge in 
each direction, and the resulting structure forms the Cartesian 
product of $m$ cyclic $\ell$-cycles.  
Thus the $\ell$-hyperlattice preserves the incidence pattern 
of the classical Cartesian lattice on $\mathbb{Z}^m$, while 
replacing pairwise edges with higher-order (multiway) 
interactions along coordinate lines.
\end{remark}
\begin{remark}
In the classical $m$-dimensional Cartesian graph on $\mathbb{Z}_\ell^m$,
each vertex has degree $2m$, since it is connected to two neighbors 
along each coordinate axis.  
In the $\ell$-hyperlattice, each axis contributes instead a single 
coordinate line through the vertex, so the hyperedge-based degree is 
$d^{(h)} = m$.  
\end{remark}

The structure of the 
$\ell$-hyperlattice in three dimensions is illustrated in Figure~\ref{fig:hyperlattice3D}, where the three hyperedges incident to a node are highlighted.
\begin{figure}[t]
    \centering
    \includegraphics[width=0.55\textwidth]{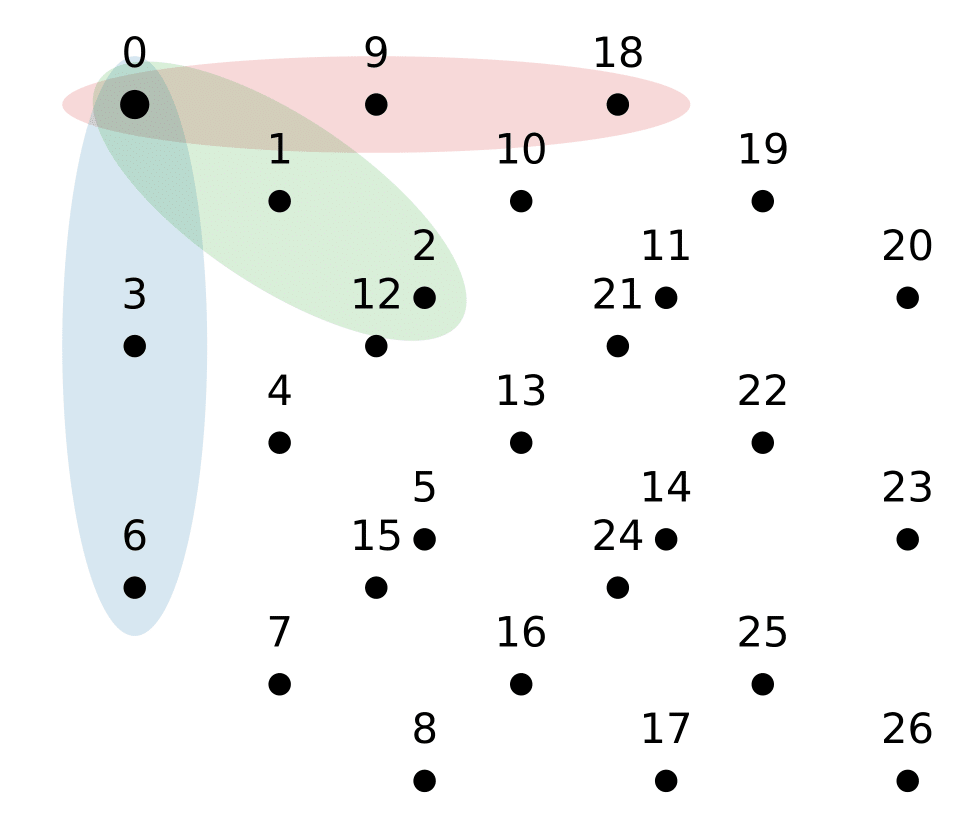}
    \caption{
        $\ell$-hyperlattice in $\mathbb{Z}_\ell^3$ with $\ell=3$.
The three hyperedges adjacent to vertex~0 are shown.
Each hyperedge contains $\ell=3$ nodes aligned along one coordinate axis and is emphasized by a colored elliptical region.
    }
    \label{fig:hyperlattice3D}
\end{figure}

\paragraph{Laplacian operators.}
For the $\ell$-hyperlattice in~$\mathbb{Z}_\ell^m$, 
the clique and hyperedge-based Laplacians take the form
\[
L_{\mathrm{cl}} = m(\ell - 1)I - A_{\mathrm{cl}},
\qquad
L_{h} = mI - \frac{1}{\ell - 1} A_{\mathrm{cl}},
\]
where $A_{\mathrm{cl}}$ denotes the adjacency matrix obtained by
clique expansion of each hyperedge.\\

We now derive the Laplacian spectra of the $\ell$-hyperlattice in~$\mathbb{Z}_\ell^m$.
\begin{proposition}[Spectrum of the $\ell$-hyperlattice in $\mathbb{Z}_\ell^m$]
Let $\mathcal{H}_\ell = (V,E)$ be the $\ell$-hyperlattice on 
$V=\mathbb{Z}_\ell^m$ and let $L_{\mathrm{cl}}$ denote its 
clique Laplacian.  
Then $L_{\mathrm{cl}}$ has exactly $m+1$ distinct eigenvalues
\[
\lambda_j = j\,\ell, 
\qquad j = 0,1,\dots,m,
\]
with multiplicities
\[
\mathrm{mult}(\lambda_j)
= \binom{m}{j}(\ell-1)^j .
\]
Equivalently,
\[
\sigma(L_{\mathrm{cl}})
= \bigcup_{j=0}^m 
\left\{ (j\ell)^{\,\binom{m}{j}(\ell-1)^j}\right\}.
\]

Moreover, the hyperedge-based Laplacian
\[
L_h 
= mI - \frac{1}{\ell-1}A_{\mathrm{cl}}
= \frac{1}{\ell-1}L_{\mathrm{cl}}
\]
has the same eigenvectors and eigenvalues
\[
\mu_j = \frac{j\,\ell}{\ell-1},
\qquad j=0,1,\dots,m,
\]
with the same multiplicities $\binom{m}{j}(\ell-1)^j$.
\end{proposition}

\begin{proof}
We work in the complex Hilbert space $\mathbb{C}^V$ with the standard inner product and use the discrete Fourier basis on the torus $V=\mathbb{Z}_\ell^m$.
For each frequency vector $k=(k_1,\dots,k_m)\in\mathbb{Z}_\ell^m$,
define
\[
\phi_k(x)
= \exp\!\left(\frac{2\pi i}{\ell}\,\langle k,x\rangle\right),
\qquad x\in V,
\]
where $\langle k,x\rangle = \sum_{i=1}^m k_i x_i$.
The family $\{\phi_k\}_{k\in\mathbb{Z}_\ell^m}$ is an orthogonal basis
of eigenvectors for all translation-invariant operators on $V$.

\smallskip

By definition of the $\ell$-hyperlattice, the clique adjacency matrix
$A_{\mathrm{cl}}$ decomposes as a sum of contributions along each
coordinate axis.  
For a fixed direction $i\in\{1,\dots,m\}$, the directional family
$\mathcal{F}_i$ consists of all coordinate lines parallel to axis $i$,
\[
L_i(x_{\hat{i}})
=
\{ (x_1,\dots,x_{i-1},t,x_{i+1},\dots,x_m) : t\in\mathbb{Z}_\ell\},
\qquad x_{\hat{i}}\in\mathbb{Z}_\ell^{\,m-1}.
\]
On each such line the clique expansion produces a complete graph
$K_\ell$ (with zero diagonal).  
Thus the adjacency contribution $A_i$ of direction $i$ acts on $f:V\to\mathbb{C}$ as
\[
(A_i f)(x)
=
\sum_{\substack{y\in V:\\ y_j=x_j \,\forall j\neq i\\ y_i\neq x_i}} f(y),
\qquad x\in V.
\]

We first compute the action of $A_i$ on a Fourier mode $\phi_k$.
For $x\in V$ we have
\[
(A_i \phi_k)(x)
=
\sum_{\substack{t\in\mathbb{Z}_\ell\\ t\neq x_i}}
\exp\!\left(
\frac{2\pi i}{\ell}
\Big(k_i t + \!\!\sum_{j\neq i} k_j x_j\Big)
\right)
=
\exp\!\left(
\frac{2\pi i}{\ell}
\sum_{j\neq i} k_j x_j
\right)
\sum_{\substack{t\in\mathbb{Z}_\ell\\ t\neq x_i}}
\exp\!\left(
\frac{2\pi i}{\ell} k_i t
\right).
\]
We distinguish two cases.

\smallskip

\emph{Case $k_i=0$.} Then the inner sum does not depend on $x_i$ and
\[
\sum_{\substack{t\in\mathbb{Z}_\ell\\ t\neq x_i}}
\exp\!\left(
\frac{2\pi i}{\ell} k_i t
\right)
=
\sum_{t\in\mathbb{Z}_\ell} 1 - 1
=
\ell-1.
\]
Hence
\[
(A_i \phi_k)(x)
=
(\ell-1)\,
\exp\!\left(
\frac{2\pi i}{\ell}
\sum_{j\neq i} k_j x_j
\right)
=
(\ell-1)\,\phi_k(x).
\]

\emph{Case $k_i\neq 0$.} Using the fact that
$\sum_{t\in\mathbb{Z}_\ell} e^{2\pi i k_i t/\ell}=0$ for $k_i\neq 0$,
we obtain
\[
\sum_{\substack{t\in\mathbb{Z}_\ell\\ t\neq x_i}}
\exp\!\left(
\frac{2\pi i}{\ell} k_i t
\right)
=
-\,\exp\!\left(
\frac{2\pi i}{\ell} k_i x_i
\right),
\]
and therefore
\[
(A_i \phi_k)(x)
=
-\,\exp\!\left(
\frac{2\pi i}{\ell}
\sum_{j=1}^m k_j x_j
\right)
=
-\,\phi_k(x).
\]

In summary, for each direction $i$,
\[
A_i \phi_k
=
\alpha_i(k)\,\phi_k,
\qquad
\alpha_i(k)
=
\begin{cases}
\ell-1, & k_i=0,\\[2pt]
-1, & k_i\neq 0.
\end{cases}
\]

Since the directional contributions are supported on disjoint cliques,
the total clique adjacency is
\[
A_{\mathrm{cl}} = \sum_{i=1}^m A_i,
\]
and $\phi_k$ is an eigenvector of $A_{\mathrm{cl}}$ with eigenvalue
\[
\alpha(k) = \sum_{i=1}^m \alpha_i(k).
\]

Let $j=j(k)$ denote the number of nonzero coordinates of $k$, i.e.
\[
j = \#\{i : k_i\neq 0\},
\qquad
0\le j\le m.
\]
Then there are $m-j$ indices $i$ with $k_i=0$ and $j$ with $k_i\neq 0$, so
\[
\alpha(k)
=
(m-j)(\ell-1) + j(-1)
=
m(\ell-1) - j\ell.
\]

\smallskip

By construction, each vertex has clique degree
$d^{(\mathrm{cl})} = m(\ell-1)$, hence
\[
L_{\mathrm{cl}}
= D_{\mathrm{cl}} - A_{\mathrm{cl}}
= m(\ell-1)I - A_{\mathrm{cl}}.
\]
Therefore $\phi_k$ is also an eigenvector of $L_{\mathrm{cl}}$ with eigenvalue
\[
\lambda(k)
=
m(\ell-1) - \alpha(k)
=
m(\ell-1) - \bigl(m(\ell-1) - j\ell\bigr)
=
j\ell.
\]
Thus the spectrum of $L_{\mathrm{cl}}$ is contained in the set
$\{j\ell : j=0,\dots,m\}$.

It remains to compute the multiplicities.  
For fixed $j\in\{0,\dots,m\}$, the eigenvalue $j\ell$ corresponds
exactly to those $k\in\mathbb{Z}_\ell^m$ with exactly $j$ nonzero
coordinates.  
There are $\binom{m}{j}$ ways to choose the $j$ nonzero positions,
and for each such position one can choose any value in
$\{1,\dots,\ell-1\}$, yielding $(\ell-1)^j$ possibilities.  
Hence
\[
\mathrm{mult}(\lambda_j)
=
\binom{m}{j}(\ell-1)^j,
\qquad
\lambda_j = j\ell,
\quad j=0,\dots,m.
\]
Since
\[
\sum_{j=0}^m \binom{m}{j}(\ell-1)^j
=
(1+\ell-1)^m
=
\ell^m
=
|V|,
\]
these multiplicities account for all eigenvectors, and the
description of $\sigma(L_{\mathrm{cl}})$ is complete.

\smallskip

Finally, by definition of the hyperedge-based Laplacian we have
\[
L_h
=
mI - \frac{1}{\ell-1}A_{\mathrm{cl}}
=
\frac{1}{\ell-1}\bigl(m(\ell-1)I - A_{\mathrm{cl}}\bigr)
=
\frac{1}{\ell-1} L_{\mathrm{cl}}.
\]
Thus $L_h$ shares the same eigenvectors $\phi_k$, and its eigenvalues are
\[
\mu_j = \frac{1}{\ell-1}\lambda_j
= \frac{j\ell}{\ell-1},
\qquad j=0,\dots,m,
\]
with the same multiplicities as above.
\end{proof}

\section{Hyperlattices in $\mathbb{Z}_\ell^m$ with hyperedges of cardinality $r$}\label{sec:2mrhyper}
We now introduce a generalization of the hyperlattice
structures considered in the previous sections.
Here the hyperedge cardinality $r$ is decoupled from the lattice side
$\ell$, and hyperedges encode local interactions along coordinate
directions on a periodic grid.\\
Let $\ell \ge 2$, $m \ge 1$ and $2 \le r \le \ell$. 
We work on the discrete $m$-dimensional torus
\[
V = \mathbb{Z}_\ell^m = 
\{ (x_1,\ldots,x_m) : x_i \in \{0,\ldots,\ell-1\} \},
\]
with indices taken modulo~$\ell$.
\begin{definition}[Hyperlattice in $\mathbb{Z}_\ell^m$ with hyperedges of size $r$]
For each coordinate direction $i \in \{1,\ldots,m\}$ and each offset 
$s \in \{0,\ldots,r-1\}$, we define an \emph{offset class} 
$\mathcal{C}_i^{(s)}$ of contiguous $r$-node hyperedges as follows.

For every choice of $x_{\neq i} \in \mathbb{Z}_\ell^{\,m-1}$ and every 
$a \in \mathbb{Z}_\ell$ with $a \equiv s \ (\mathrm{mod}\ r)$, 
consider the hyperedge
\[
H_i^{(s)}(a; x_{\neq i})
=
\bigl\{ (x_1,\ldots,x_m) \in V : 
x_j = (x_{\neq i})_j \ \text{for } j \neq i,\ 
x_i \in [a, a + r - 1]_\ell \bigr\},
\]
where $[a,a+r-1]_\ell$ denotes $r$ consecutive indices modulo $\ell$.
The class $\mathcal{C}_i^{(s)}$ is the collection of all such hyperedges.

The hyperlattice with hyperedges of size $r$ is the hypergraph
\[
\mathcal{H}_r = (V,\mathcal{E}), 
\qquad
\mathcal{E} = 
\bigcup_{i=1}^m \ \bigcup_{s=0}^{r-1} \mathcal{C}_i^{(s)}.
\]
\end{definition}

\paragraph{Degrees and Laplacian operators.}
By construction, for each direction $i$ and each vertex $v \in V$ there are
exactly $r$ hyperedges in the families $\mathcal{C}_i^{(s)}$ containing $v$
(one for each offset), so that every vertex is incident to
\[
d^{(h)}(v) = mr
\]
hyperedges. 
Since each hyperedge has cardinality $r$, each contributes $r-1$ 
clique-neighbours per vertex, and thus
\[
d^{(\mathrm{cl})}(v) = mr(r-1),
\qquad \forall v \in V.
\]
Accordingly, the clique and hyperedge-based Laplacians are
\[
L_{\mathrm{cl}} = mr(r-1)I - A_{\mathrm{cl}},
\qquad
L_h = mrI - \frac{1}{r-1} A_{\mathrm{cl}}.
\]
\begin{remark}
Along each coordinate direction, the construction generates $r$ 
translation-invariant offset classes of contiguous $r$-node hyperedges,
corresponding to the $r$ possible offsets of the sliding windows.
These classes partition the set of all $r$-windows along that axis.
Different coordinate directions intersect in exactly one vertex for 
each choice of lines, so the resulting hypergraph preserves the 
Cartesian incidence pattern of the classical $m$-dimensional lattice 
graph while encoding higher-order interactions of finite range $r$.
For $r=\ell$ the construction reduces to the $\ell$-hyperlattice
introduced in the previous subsection.
\end{remark}

To illustrate the sliding construction, we decompose the periodic
hyperlattice in~$\mathbb{Z}_\ell^m$ with hyperedges of size~$r$
into its $r$ directional sub-families, 
each corresponding to one possible offset of the sliding windows
along every coordinate axis.
Figure~\ref{fig:2mrhyperlattice} shows the case $m = 2$, $\ell = 6$, and $r = 3$,
where each sub-family (one per column of the figure) contains
the horizontal and vertical hyperedges sharing the same offset.
Every node belongs to exactly two hyperedges in each panel
(one per direction), so that the union of the three sub-families
accounts for all $m r = 6$ hyperedges incident to every node in
the complete hyperlattice.

\begin{figure}[ht]
    \centering
    \includegraphics[width=.32\linewidth]{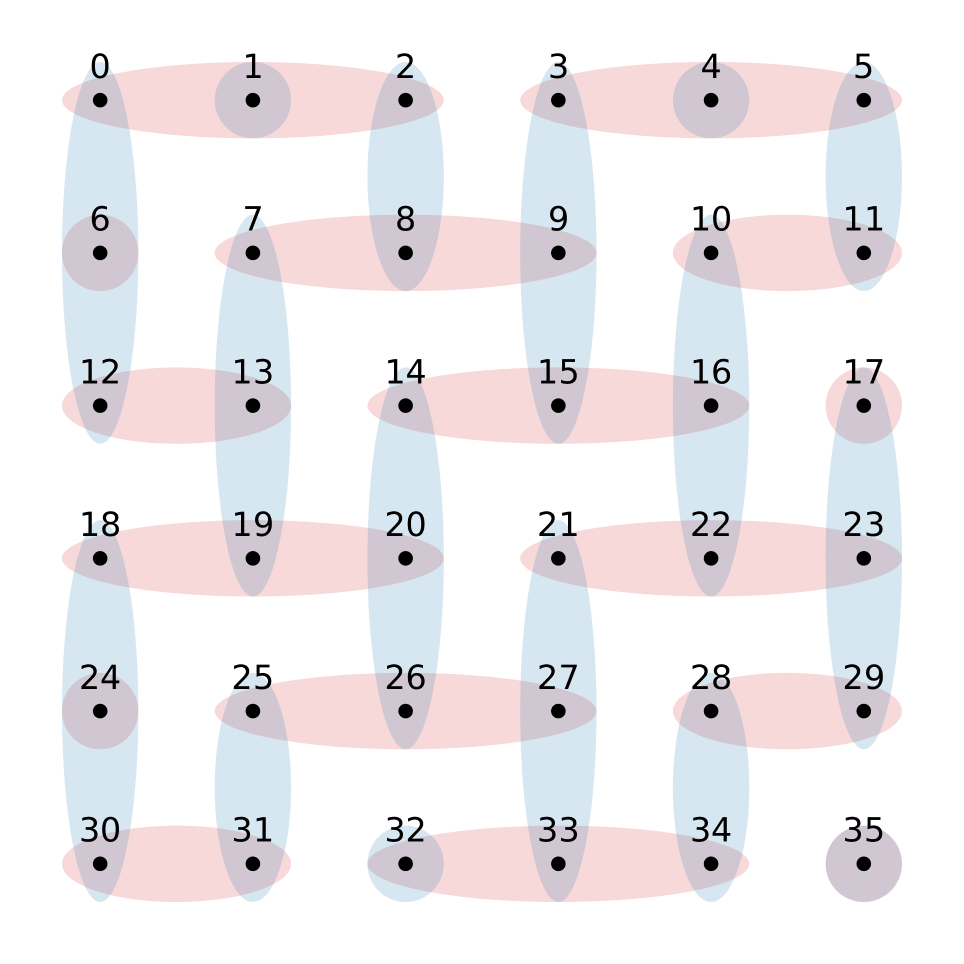}
    \includegraphics[width=.32\linewidth]{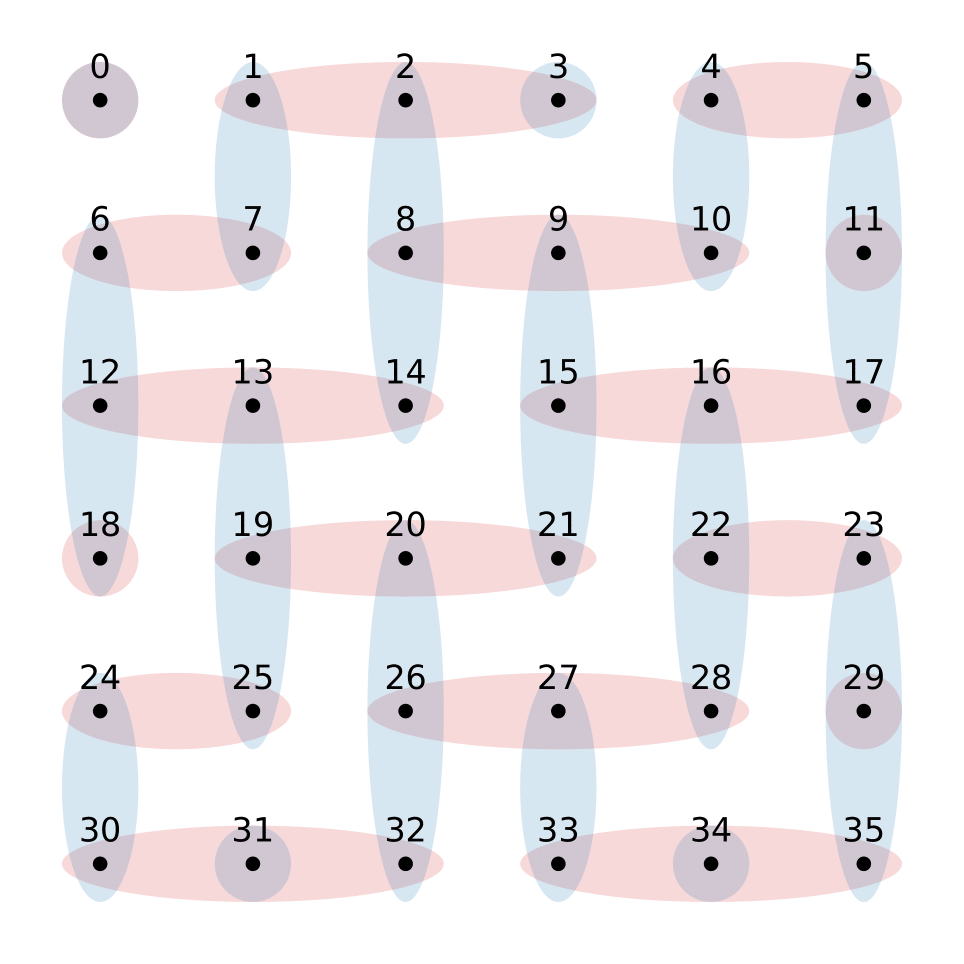}
    \includegraphics[width=.32\linewidth]{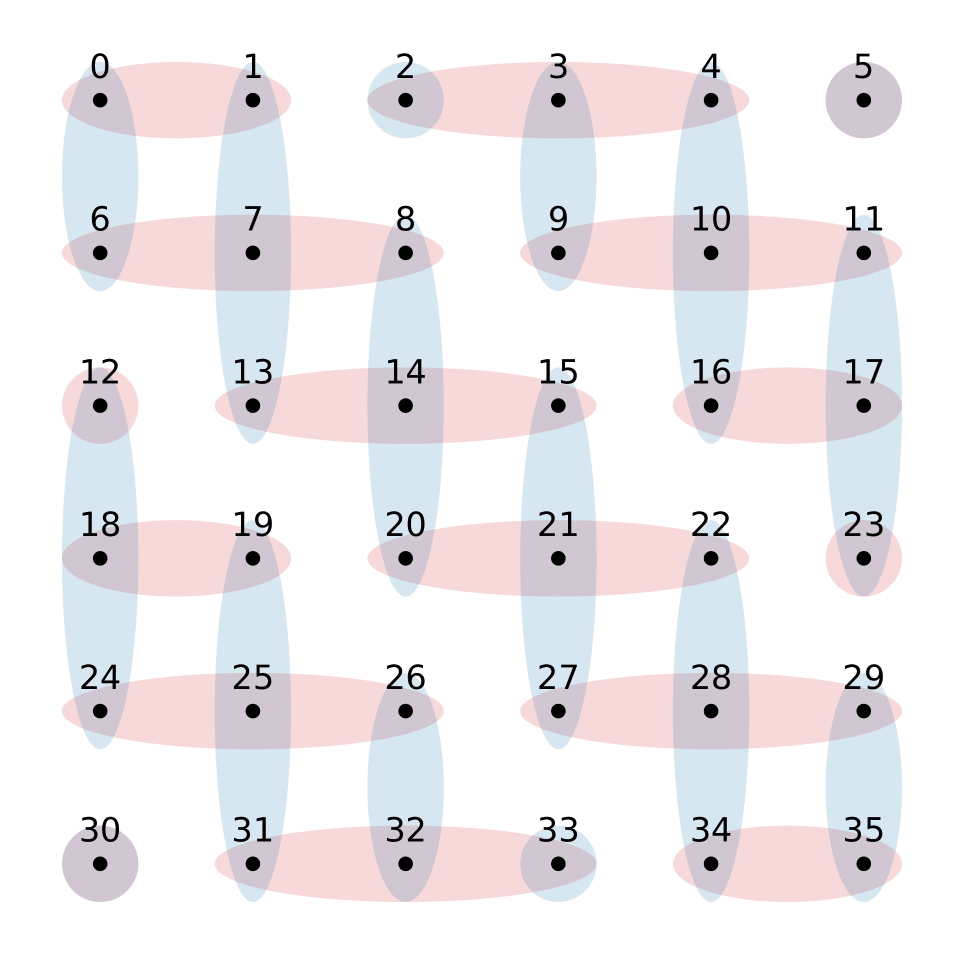}
    \caption{
The three panels illustrate the hyperlattice in~$\mathbb{Z}_\ell^m$
with hyperedges of size~$r$, for $m = 2$, $\ell = 6$, and $r = 3$.
The structure is decomposed into its $r$ directional sub-families,
corresponding to the three possible sliding offsets 
($\mathrm{shift}=0,1,2$).
Each sub-figure displays the horizontal (red) and vertical (blue)
hyperedges sharing the same offset, so that every node belongs
to exactly two hyperedges (one per direction) in each panel.
Together, the three sub-families account for the full set of
$mr = 6$ hyperedges incident to every node in the complete hyperlattice.
}
\label{fig:2mrhyperlattice}
\end{figure}

\begin{theorem}[Spectrum of the hyperlattice in $\mathbb{Z}_\ell^m$ with hyperedges of size $r$]
Let $\mathcal{H}_r = (V,\mathcal{E})$ be the hyperlattice 
on the discrete torus $V = \mathbb{Z}_\ell^m$, 
with hyperedges given by all contiguous $r$-node windows
along each coordinate direction.
Let $A_{\mathrm{cl}}$ be the associated clique adjacency matrix.

Then $A_{\mathrm{cl}}$ is block-circulant and diagonalizable
in the Fourier basis
\[
\varphi_k(x)
= \exp\!\left(\frac{2\pi i}{\ell}\langle k,x\rangle\right),
\qquad
k=(k_1,\dots,k_m)\in\mathbb{Z}_\ell^m.
\]
For each $k$, the corresponding eigenvalue of $A_{\mathrm{cl}}$ is
\[
\alpha(k)
= \sum_{i=1}^m \alpha_{1D}(k_i),
\qquad
\alpha_{1D}(k_i)
= 2\sum_{d=1}^{r-1} (r-d)\,
   \cos\!\left(\frac{2\pi k_i d}{\ell}\right),
\]
with $\alpha_{1D}(0) = r(r-1)$.
Hence the clique and hyperedge-based Laplacian eigenvalues are
\[
\lambda_{\mathrm{cl}}(k)
= mr(r-1) - \alpha(k),
\qquad
\lambda_h(k)
= mr - \frac{\alpha(k)}{r-1}.
\]

Each distinct multi-index $k\in\mathbb{Z}_\ell^m$
corresponds to one eigenmode, so the multiplicity of $\lambda_{\mathrm{cl}}(k)$
is the number of distinct vectors yielding the same value of $\alpha(k)$.
The constant mode $k=0$ gives $\lambda_{\mathrm{cl}}(0)=\lambda_h(0)=0$ with
multiplicity one, confirming that $\mathcal{H}_r$ is connected.
\end{theorem}

\begin{proof}
We first analyze the one-dimensional case, $m=1$.
Let $V = \mathbb{Z}_\ell$, and let each hyperedge be a sliding window
of $r$ consecutive vertices,
\[
H_a = \{a, a+1, \dots, a+r-1\}_\ell,
\qquad a \in \mathbb{Z}_\ell.
\]
The clique adjacency matrix $A_{\mathrm{cl}}^{(1D)}$ has entries
\[
(A_{\mathrm{cl}}^{(1D)})_{xy}
= \#\{\,a : x,y \in H_a\,\},
\]
which depend only on the cyclic distance
$d(x,y) = \min\{|x-y|,\,\ell-|x-y|\}$.
For a fixed distance $d \in \{1,\dots,\ell-1\}$,
the number of hyperedges containing both $x$ and $y$
equals the number of windows of length $r$ that cover
two points at distance~$d$, namely
\[
w(d)
= \max\{0,\, r-d\}.
\]
Therefore $A_{\mathrm{cl}}^{(1D)}$ is a real symmetric circulant matrix with
entries
\[
w(0)=0,
\qquad
w(d)=w(\ell-d)=r-d
\quad\text{for }1\le d\le r-1.
\]
The clique-degree in one dimension is
\[
d^{(\mathrm{cl})}_{1D}
= \sum_{d=1}^{\ell-1} w(d)
= 2\sum_{d=1}^{r-1} (r-d)
= r(r-1).
\]

Since $A_{\mathrm{cl}}^{(1D)}$ is circulant, it is diagonalized by the discrete
Fourier basis
\[
\phi_k(x) = \exp\!\left(\frac{2\pi i k x}{\ell}\right),
\qquad k=0,\dots,\ell-1.
\]
The corresponding eigenvalues are
\[
\alpha_{1D}(k)
= \sum_{d=0}^{\ell-1} w(d)
  \exp\!\left(\frac{2\pi i k d}{\ell}\right)
= 2\sum_{d=1}^{r-1} (r-d)
  \cos\!\left(\frac{2\pi k d}{\ell}\right).
\tag{$\ast$}
\]

To obtain a closed form, we use the classical trigonometric identity
(valid for any $r\ge2$ and $\theta\ne0$)
\[
\sum_{d=1}^{r-1} (r-d)\cos(d\theta)
= \frac{\sin^2\!\left(\frac{r\theta}{2}\right)}
       {2\sin^2\!\left(\frac{\theta}{2}\right)}
  - \frac{r}{2}.
\]
Substituting $\theta = \tfrac{2\pi k}{\ell}$ in $(\ast)$ gives
\[
\alpha_{1D}(k)
= 
     \frac{\sin^2\!\left(\frac{r\pi k}{\ell}\right)}
          {\sin^2\!\left(\frac{\pi k}{\ell}\right)}
     - r
   ,
\qquad
k = 0,\dots,\ell-1,
\]
where the value at $k=0$ is understood by continuity as
$\alpha_{1D}(0)=r(r-1)$.

\medskip\noindent
\emph{Multidimensional case.}
For $m>1$, the hyperlattice in $\mathbb{Z}_\ell^m$ with hyperedges of size~$r$
is defined as the union of $m$ independent directional families,
each acting along one coordinate axis.
Let $A_{\mathrm{cl}}^{(i)}$ denote the clique adjacency acting along
direction~$i$.
These operators commute and share the same Fourier eigenbasis
\[
\varphi_k(x)
= \exp\!\left(\frac{2\pi i}{\ell}\langle k,x\rangle\right),
\qquad k=(k_1,\dots,k_m)\in\mathbb{Z}_\ell^m.
\]
Since $A_{\mathrm{cl}}^{(i)}$ acts as the one-dimensional operator along the
$i$th coordinate, we have
\[
A_{\mathrm{cl}}^{(i)}\varphi_k
= \alpha_{1D}(k_i)\,\varphi_k.
\]
Summing over $i=1,\dots,m$ yields
\[
A_{\mathrm{cl}}\varphi_k
= \left(\sum_{i=1}^m \alpha_{1D}(k_i)\right)\varphi_k,
\qquad
\alpha(k)
= \sum_{i=1}^m \alpha_{1D}(k_i).
\]
The clique-degree is $d^{(\mathrm{cl})} = m r(r-1)$, consistent with
$\alpha(0) = d^{(\mathrm{cl})}$.
Therefore the Laplacian eigenvalues are
\[
\lambda_{\mathrm{cl}}(k)
= m r(r-1) - \alpha(k),
\qquad
\lambda_h(k)
= m r - \frac{\alpha(k)}{r-1}.
\]
Since $\alpha_{1D}(k_i) < r(r-1)$ for all $k_i\neq0$,
we have $\alpha(k)<m r(r-1)$ for all $k\neq0$,
so that all non-constant modes have strictly positive eigenvalues.
Thus the constant mode $k=0$ is the unique zero eigenvector, and the
hyperlattice is connected.
\end{proof}

\begin{remark}[Periodic and open-boundary variants]
The construction above assumes periodic boundary conditions,
so that all sliding windows wrap around the domain and every node
participates in exactly $r$ hyperedges per direction.
This periodicity ensures full translation invariance and a uniform
vertex degree across the hyperlattice.

In the non-periodic setting, several open-boundary variants can be defined.
One possibility is to truncate the sliding windows that extend beyond the
domain, thereby allowing boundary hyperedges of reduced cardinality
($1 \le |e| \le r$);
another is to remove such windows altogether, keeping only those fully
contained in the finite region so that all hyperedges have the same size~$r$.
All these non-periodic variants constitute natural finite-domain analogues of
the periodic hyperlattice in~$\mathbb{Z}_\ell^m$ with hyperedges of size~$r$,
and play the same role as open lattices in classical statistical physics.
\end{remark}

To obtain an open-boundary analogue of the hyperlattice whose spectral
structure mirrors that of open lattice graphs in any dimension
(e.g., paths in one dimension, planar grids in two dimensions, and
their higher-dimensional generalizations), 
we seek a construction that preserves the Toeplitz property of the
clique adjacency, so that the eigenvalues can be computed analytically
through trigonometric modes.

In the natural non-periodic models discussed above, this property is
lost because the number of sliding windows containing a given pair of
vertices depends on their position relative to the boundary.
To restore translation invariance and obtain a symmetric Toeplitz
structure, we adopt the following definition.

\begin{definition}[Open Toeplitz hyperlattice]\label{def:toep}
Let $\ell \ge r \ge 2$ and $m \ge 1$.
We define the \emph{open Toeplitz hyperlattice} in the finite domain
$V = \{1,\dots,\ell\}^m$ as the hypergraph
$\mathcal{H}_r^{\mathrm{T}} = (V, \mathcal{E})$
whose hyperedges are obtained as the intersections between the
$r$-node sliding windows on the infinite lattice~$\mathbb{Z}^m$
and the finite vertex domain~$V$, i.e.
\[
\mathcal{E} = 
\bigl\{\, H_i(a; x_{\neq i}) \cap V :
i \in \{1,\dots,m\},\ 
a \in \mathbb{Z},\ 
x_{\neq i} \in \mathbb{Z}_\ell^{m-1}
\,\bigr\},
\]
where
\[
H_i(a; x_{\neq i}) =
\{ (x_1,\dots,x_m) \in \mathbb{Z}^m :
x_j = (x_{\neq i})_j \text{ for } j \neq i,\
x_i \in [a,a+r-1]_\mathbb{Z} \}.
\]
All non-empty intersections are retained, so that near the boundary
some hyperedges may have cardinality $|e| < r$.
This reflective construction ensures that the number of hyperedges
containing a given pair of vertices depends only on their distance
along each coordinate direction.
\end{definition}

In this way, vertices close to the boundary effectively receive the 
same number of hyperedge contributions as those in the interior,
as if the lattice were extended beyond the domain.
This can be interpreted as a \emph{reflective boundary condition} 
that maintains constant pairwise weights along each coordinate axis,
analogous to the transition from periodic to open lattices in
classical graph theory.

Under this definition, the pairwise weights depend only on the 
distance along each coordinate direction, yielding a multi-Toeplitz
adjacency matrix that can be diagonalized by a separable discrete
sine transform, as stated in the following proposition.

\begin{proposition}[Spectrum of the open Toeplitz hyperlattice in $\{1,\dots,\ell\}^m$]
Let $\mathcal{H}_r^{\mathrm{T}} = (V,\mathcal{E})$ be the open Toeplitz hyperlattice
on the finite domain $V = \{1,\dots,\ell\}^m$ as in Definition~\ref{def:toep}.
Let $A_{\mathrm{cl}}$ be its clique adjacency matrix.

Then $A_{\mathrm{cl}}$ can be written as
\[
A_{\mathrm{cl}} = \sum_{i=1}^m A_{\mathrm{cl}}^{(i)},
\]
where each $A_{\mathrm{cl}}^{(i)}$ is a real symmetric Toeplitz operator
acting along the $i$-th coordinate direction with bandwidth $r-1$.
In particular, $A_{\mathrm{cl}}$ is a multi-level Toeplitz matrix and is
diagonalizable in a separable trigonometric basis.

More precisely, let $\alpha_{1D}(k)$ denote the one-dimensional eigenvalues
of the Toeplitz kernel
\[
w(d) =
\begin{cases}
r-d, & 1 \le d \le r-1,\\[2pt]
0, & d \ge r,
\end{cases}
\]
that is
\[
\alpha_{1D}(k)
= 2 \sum_{d=1}^{r-1} (r-d)
    \cos\!\left(\frac{d k \pi}{\ell+1}\right),
\qquad k = 1,\dots,\ell.
\]
For each multi-index $k = (k_1,\dots,k_m) \in \{1,\dots,\ell\}^m$, define
\[
\alpha(k) = \sum_{i=1}^m \alpha_{1D}(k_i).
\]
Then:
\begin{enumerate}
\item The eigenvalues of the clique adjacency matrix are
\[
\spec(A_{\mathrm{cl}})
= \{\, mr(r-1) \,\} \,\cup\,
\{\, \alpha(k) : k \in \{1,\dots,\ell\}^m \},
\]
where $mr(r-1)$ corresponds to the constant eigenvector, and
each $\alpha(k)$ is associated with a separable trigonometric mode.

\item The clique Laplacian
\[
L_{\mathrm{cl}} = mr(r-1)I - A_{\mathrm{cl}}
\]
has eigenvalues
\[
\lambda_{\mathrm{cl}}(k)
= mr(r-1) - \alpha(k),
\qquad k \in \{1,\dots,\ell\}^m,
\]
together with a simple zero eigenvalue corresponding to the constant mode.

\item The hyperedge-based Laplacian
\[
L_h = mrI - \frac{1}{r-1} A_{\mathrm{cl}}
\]
has eigenvalues
\[
\lambda_h(k)
= mr - \frac{1}{r-1}\,\alpha(k),
\qquad k \in \{1,\dots,\ell\}^m,
\]
again with a simple zero eigenvalue obtained from the constant vector.
\end{enumerate}
\end{proposition}

\begin{proof}
By Definition~\ref{def:toep}, hyperedges are obtained
by intersecting the $r$-node sliding windows on $\mathbb{Z}^m$
with the finite domain $V = \{1,\dots,\ell\}^m$, along each coordinate
direction $i=1,\dots,m$.
Fix $i$ and a choice of $(x_1,\dots,x_{i-1},x_{i+1},\dots,x_m)$.
Along the $i$-th coordinate, this reduces to the one-dimensional
construction on $\{1,\dots,\ell\}$ with the Toeplitz kernel $w(d)$
introduced above.

Thus the clique adjacency decomposes as
\[
A_{\mathrm{cl}} = \sum_{i=1}^m A_{\mathrm{cl}}^{(i)},
\]
where each $A_{\mathrm{cl}}^{(i)}$ acts non-trivially only along the
$i$-th coordinate, with entries
\[
(A_{\mathrm{cl}}^{(i)})_{xy}
= w\bigl(|x_i - y_i|\bigr)
\prod_{j \neq i} \mathbf{1}_{\{x_j = y_j\}}.
\]
Each $A_{\mathrm{cl}}^{(i)}$ is a symmetric Toeplitz operator of
bandwidth $r-1$ along direction~$i$.

In one dimension, the standard spectral theory for real symmetric
Toeplitz matrices with kernel $w(d)$ shows that $A_{\mathrm{cl}}^{(i)}$
is diagonalized by trigonometric modes with eigenvalues
$\alpha_{1D}(k_i)$ as above.
Since the operators $\{A_{\mathrm{cl}}^{(i)}\}_{i=1}^m$ commute and act
independently on each coordinate, their joint eigenvectors are given by
separable products of one-dimensional modes:
\[
\psi_k(x)
= \prod_{i=1}^m \sin\!\left(\frac{x_i k_i \pi}{\ell+1}\right),
\qquad
k = (k_1,\dots,k_m) \in \{1,\dots,\ell\}^m,
\]
or equivalently by an appropriate discrete cosine/sine basis, depending
on the chosen boundary normalization.
For such a mode we have
\[
A_{\mathrm{cl}}^{(i)} \psi_k = \alpha_{1D}(k_i)\, \psi_k,
\]
and therefore
\[
A_{\mathrm{cl}} \psi_k
= \left( \sum_{i=1}^m \alpha_{1D}(k_i) \right) \psi_k
= \alpha(k)\, \psi_k.
\]

The constant vector $\mathbf{1}$ is also an eigenvector of $A_{\mathrm{cl}}$,
since every row sum is
\[
\sum_{y \in V} (A_{\mathrm{cl}})_{xy}
= \sum_{i=1}^m \sum_{d=1}^{r-1} 2(r-d)
= m r(r-1),
\]
independently of $x$.
Thus $\mathbf{1}$ has eigenvalue $mr(r-1)$.
This yields the full spectrum of $A_{\mathrm{cl}}$ as stated.

Finally, the formulas for the eigenvalues of $L_{\mathrm{cl}}$ and $L_h$
follow directly from the relations
\[
L_{\mathrm{cl}} = mr(r-1)I - A_{\mathrm{cl}},
\qquad
L_h = mrI - \frac{1}{r-1} A_{\mathrm{cl}},
\]
with the constant vector mapped to eigenvalue $0$ in both cases.
\end{proof}
This shows that the open Toeplitz hyperlattice admits a fully explicit,
separable spectral representation, directly comparable to that of discrete
Laplacians on finite paths, grids, and their higher-dimensional analogues.
In contrast, the more natural truncated and removed-window variants lose the
Toeplitz structure at the boundary and can only be analyzed as perturbations
of the periodic case.

\begin{corollary}
For $m=1$ and $r=2$, the open Toeplitz hyperlattice reduces to the path
graph on $\{1,\dots,\ell\}$, and the above formulas recover the classical
eigenvalues $2 - 2\cos\bigl(\frac{k\pi}{\ell+1}\bigr)$.
For $m=2$ and $r=2$, one recovers the standard open grid, with Laplacian
eigenvalues given by the sum of the one-dimensional ones.
\end{corollary}
\begin{remark}
In the special case $r=2$, all hyperedges have size two and the
hyperlattice reduces to a standard graph.
Consequently, the clique and hyperedge-based Laplacians coincide,
recovering the usual graph Laplacian on paths, grids, and their
higher-dimensional generalizations.
For $r>2$, the two operators differ, encoding distinct notions of
higher-order diffusion on the same combinatorial structure.
\end{remark}

\section{Concluding remarks}\label{sec:ising}
In this work we generalized the hypergraph lattice introduced by Andreotti and Mulas~\cite{AndreottiMulas2022Signless} both in its spatial dimension and in the cardinality of its hyperedges, developing a unified framework for higher-order lattice structures on the discrete torus~$\mathbb{Z}^m_\ell$. These {hyperlattices} extend classical graph lattices to the regime of higher-order interactions. For each model in this family we derived in closed form the spectra of the clique-based and the hyperedge-based Laplacians, exploiting translation invariance to obtain fully separable Fourier eigenmodes in the periodic setting. We also introduced a Toeplitz-type open-boundary analogue whose spectrum admits an explicit trigonometric representation, thereby providing a natural higher-order counterpart of the classical spectral formulas for open grids. In this way, several foundational results of spectral graph theory naturally generalize to higher-order structures when multiway interactions preserve a suitable form of directional homogeneity.\\
The spectral structure emerging from these constructions illustrates how the geometry of multiway interactions influences diffusion and collective behaviour. The clique Laplacian reinforces pairwise adjacency induced by hyperedges, whereas the hyperedge-based Laplacian maintains the higher-order incidence pattern. Even in highly regular, isotropic settings these operators yield markedly different diffusion scales, showing that they encode complementary aspects of higher-order connectivity. The explicit formulas obtained for periodic and open hyperlattices clarify the roles played by spatial dimension, number of directional families, and hyperedge cardinality, thereby providing a transparent analytical laboratory for exploring higher-order diffusion processes.\\
In analogy with the role played by regular lattices in the classical Ising model, where translational symmetry enables exact or closed-form characterizations of phase transitions, our hyperlattices provide an analytically tractable substrate for higher-order spin systems. While the higher-order Ising model introduced by Robiglio et al.~\cite{Robiglio2025HigherOrderIsing} is formulated on arbitrary hypergraphs, without assuming any underlying spatial regularity, the quadratic term of its Hamiltonian coincides with the hyperedge-based Laplacian studied here. This positions regular hyperlattices as the natural higher-order counterpart of classical lattices, offering a clean setting in which stability, criticality and relaxation phenomena may be analysed in exact spectral form.\\
More broadly, the hyperlattices introduced in this paper provide an analytically tractable class of higher-order substrates, bridging classical lattice models and the emerging field of higher-order network dynamics. Their explicit spectral characterization paves the way for systematic studies of diffusion, synchronization and transport in multiway systems, and offers a foundation for constructing solvable benchmarks in statistical physics, nonlinear dynamics and machine learning. Future directions include the incorporation of weighted or anisotropic hyperedges, heterogeneous interaction ranges, and nonlinear processes whose linearization reduces to the operators analysed here. Each of these extensions would further shed light on the mathematical structure underlying collective behaviour in systems governed by intrinsically higher-order interactions.

\bibliographystyle{acm}
\bibliography{embeddedbib}
\end{document}